\documentclass[11pt,dvips,twoside,letterpaper]{article}
\usepackage[usenames]{color}

\usepackage{pslatex}
\usepackage{fancyhdr}
\usepackage{graphicx}
\usepackage{geometry}

\def\figurename{Figure} % Replace the colon that normally appears after the Figure number by a period.
\makeatletter
\renewcommand{\fnum@figure}[1]{\figurename~\thefigure.}
\makeatother

\def\tablename{Table} % Replace the colon that normally appears after the Figure number by a period.
\makeatletter
\renewcommand{\fnum@table}[1]{\tablename~\thetable.}
\makeatother

\usepackage{amsmath}
\usepackage{amssymb}
\usepackage{amsfonts}
\usepackage{amsthm,amscd}

\newtheorem{theorem}{Theorem}[section]
\newtheorem{lemma}[theorem]{Lemma}
\newtheorem{corollary}[theorem]{Corollary}
\newtheorem{proposition}[theorem]{Proposition}
\theoremstyle{definition}
\newtheorem{definition}[theorem]{Definition}

\theoremstyle{remark}
\newtheorem{remark}[theorem]{Remark}

\numberwithin{equation}{section}

\begin{document}
\vskip 0.4in
\title{\bfseries\scshape{Homogenization of a nonlinear elliptic problem with
large nonlinear potential}}

\author{\bfseries\scshape Hermann Douanla\thanks{E-mail address: \tt{douanla@chalmers.se}}\\
Department of Mathematical Sciences \\Chalmers University of
Technology\\ Gothenburg, SE-41296, Sweden\\\texttt{}\\
\bfseries\scshape Nils Svanstedt\thanks{E-mail address: \tt{nilss@chalmers.se}}\\
Department of Mathematical Sciences \\ University of Gothenburg\\
Gothenburg, SE-41296, Sweden}

\date{}
\maketitle \thispagestyle{empty}

\begin{abstract} \noindent
Homogenization is studied for a nonlinear elliptic boundary-value
problem with a large nonlinear potential. More specifically we are
interested in the asymptotic behavior of a sequence of p-Laplacians
of the form
$$
-\text{div}\left(a(\frac{x}{\varepsilon})|Du_\varepsilon|^{p-2}Du_\varepsilon\right)
+\frac{1}{\varepsilon}V(\frac{x}{\varepsilon})|u_\varepsilon|^{p-2}u_\varepsilon=f.
$$
It is shown that, under a centring condition on the potential $V$,
there exists a two-scale homogenized system with solution $(u, u_1)$
such that the sequence $u_\varepsilon$ of solutions converges weakly
to $u$ in $W^{1,p}$ and the gradients $D_x u_\varepsilon$ two-scale
converges weakly to $D_x u+ D_y u_1$ in $L^p$, respectively. We
characterize the limit system explicitly by means of two-scale
convergence and a new convergence result.
\end{abstract}


\vspace{.08in}


\section{Introduction}\label{s1}
In this article, we are interested in the asymptotic behavior (as
$\varepsilon\to 0$) of the nonlinear boundary-value problem

\begin{equation} \label{eq1}
\left\{\begin{aligned}&
-\text{div}\left(a(\frac{x}{\varepsilon})|Du_\varepsilon|^{p-2}Du_\varepsilon\right)
+\frac{1}{\varepsilon}V(\frac{x}{\varepsilon})|u_\varepsilon|^{p-2}u_\varepsilon=f\quad
\text{in }\quad
\Omega\\&u_\varepsilon=0\qquad\qquad\qquad\qquad\qquad\qquad\qquad\qquad\quad\quad\text{
on }
\partial\Omega,
\end{aligned}\right.
\end{equation}
where we assume that $2\leq p<N$ and that $\Omega$ is a bounded open
set in $\mathbb{R}^N$. We assume that the matrix $a(y)$ is positive
definite with entries that are bounded and periodic with period
$Y=(0,1)^N$ and that the function $V(y)$ is smooth, periodic and has
vanishing mean value on $Y$. The existence theory is valid for
$1<p<\infty$, but for the asymptotic analysis we restrict ourselves
to the case $2\leq p<N$ in this work. For the case $N\leq p<\infty$
one would need to use Morrey’s inequality in the estimates below
instead of the Poincar\'{e} inequality. Also for the case $1<p<2$
other function spaces are needed. Problem \ref{eq1} is a nonlinear
Schr\"{o}dinger type equation with large potential. For the linear
case, $p=2$, homogenization results for Schr\"{o}dinger equations
with large potential term of order $\frac{1}{\varepsilon^2}$ is
studied by Allaire and Piatnitski in \cite{AP1}, where they use a
factorization principle to handle the large potential. We include
the case $p=2$ in this work since the homogenization result is also
valid in this case.

A nonlinear time-dependent reaction-diffusion problem with linear
elliptic term and nonlinear reaction term of order
$\frac{1}{\varepsilon}$ has been recently studied by Allaire and
Piatnitski in \cite{AP2}.

Homogenization of the p-Laplace equation is by now standard, see
e.g. \cite{AAPP} and the references therein. The novelty in this
work is the presence of a large potential term. The rate
$\frac{1}{\varepsilon}$ is motivated by the fact that this scaling
yields a local problem which has a potential term (see the second
equation in (\ref{eq2}) and thus it has the same structure as
(\ref{eq1}). In the linear case this can be illustrated by a formal
two-scale asymptotic analysis. The effect of the large potential
term is also very interesting. In the linear case the homogenized
equation is a convection-diffusion equation, so the limit equation
is of different type than the original equation. Without the
$\frac{1}{\varepsilon}$-scaling the homogenized equation is of the
same type as the original equation. With nonlinear potential the
homogenized equation has the lower order term $V(y) F'(u)u_1$, where
$u_1=u_1(u,D_x u)$ so it can contain both convection type and
potential type contributions. This effect is already observed and
thoroughly discussed in \cite{AP2}. The
$\frac{1}{\varepsilon}$-scaling is also motivated by the fact that
it provides an a priori estimate independent of $\varepsilon$, see
Lemma \ref{l1}, which is needed in order to prove the homogenization
result for (\ref{eq1}).

A key result in this work, which significantly simplifies the
homogenization in our approach, is the new compactness result, Lemma
\ref{l2}, which is used to handle the nonlinear potential term when
$p>2$. The result of Lemma \ref{l2} is dependent on the centring
condition (mean value zero over the period) of the potential $V$
which allows us to apply a Fredholm alternative argument. Dropping
the centring condition one way to overcome the difficulty of the
large term could be to try to establish a factorization principle
argument like in \cite{AP1}.

In this work we prove that, as $\varepsilon\to 0$, the solution
$u_\varepsilon$ to (\ref{eq1}) satisfies

$$
u_\varepsilon\to u\quad \text{in}\quad L^p(\Omega),
$$
$$
D_x u_\varepsilon \to D_x u+D_y u_1\quad\text{in}\quad
L^p(\Omega)\quad \text{two-scale weakly},
$$
where $(u,u_1)\in W^{1,p}_0(\Omega)\times
L^p(\Omega;W^{1,p}_{per}(Y))$ solves the two-scale homogenized
system

\begin{equation} \label{eq2}
\left\{\begin{aligned}& -\text{div}_x\left(a(y)|D_x u+D_y
u_1|^{p-2}(D_x u+D_y u_1)\right)+V(y)F'(u)u_1=f,\\&
-\text{div}_y\left(a(y)|D_x u+D_y u_1|^{p-2}(D_x u+D_y
u_1)\right)+V(y)F(u)=0,\\&u=0\quad\text{ on }
\partial\Omega,
\end{aligned}\right.
\end{equation}
where $F(u)=|u|^{p-2}u$ and for $p>2$, $F'(u)=(p-1)u^{p-2}$ for
$u\geq 0$ and $F'(u)=(1-p)|u|^{p-3}u$ for $u< 0$.

\section{Pseudomonotone operators}\label{s2}
In this section, we introduce a class of operators, pseudomonotone
operators, in order to prove existence of solution to

\begin{equation} \label{eq3}
\left\{\begin{aligned}& -\text{div}\left(a|Du|^{p-2}Du\right)
+V|u|^{p-2}u=f\quad \text{in }\quad
\Omega\\&u=0\qquad\qquad\qquad\qquad\qquad\qquad\quad\text{ on }
\partial\Omega,
\end{aligned}\right.
\end{equation}
where we allow the potential $V$ to be negative but bounded from
below.

We prove the existence of weak solutions $u\in W^{1,p}_0(\Omega)$ to
(\ref{eq3}), where $\Omega$ is a bounded, open set in
$\mathbb{R}^N$, $p$ is a real number $2\leq p<N$, the matrix $a\in
L^\infty(\mathbb{R}^N)$ is strictly positive definite, i.e. there
exists a constant $M>0$ such that
$$
\sum_{i,j=1}^N a_{ij}(x)\xi_j\xi_i\geq M|\xi|^2,
$$
for all $\xi\in \mathbb{R}^N$ and a.e. $x\in \mathbb{R}^N$ and where
the potential $V\in L^\infty(\Omega)$.

We emphasize that monotonicity theory does not apply here, since the
term corresponding to $F(u)=V|u|^{p-2}u$ does not necessarily
satisfy the monotonicity condition
$$
(F(u)-F(v))(u-v)\geq 0.
$$

We recall some basic facts for pseudomonotone operators. Most of the
details can be found in \cite{Zeidler}.

In the sequel the letter $C$ with or without subindex denotes a
generic constant whose value might change from one line to another.
\begin{definition}\label{d1}
Let $X$ be a reflexive Banach space with dual space $X^*$. An
operator $T: X\to X^*$ is called pseudomonotone if $ u_n\to u,
\text{ weakly in } X \text{ as } n \to +\infty $ and $ \lim_{n\to
+\infty}\langle Tu_n, u_n-u\rangle\leq 0 $ imply
$$
\langle Tu_n, u-w\rangle\leq\lim_{n\to +\infty}\langle Tu_n,
u_n-w\rangle
$$
for all $w\in X$.
\end{definition}

\begin{definition}\label{d2}
An operator $T: X\to X^*$ is called strongly continuous if $u_n\to
u$ weakly as $n\to+\infty$ implies that $Tu_n\to Tu$ in $X^*$.
\end{definition}

\begin{proposition}\label{p1}
Let $T_1,T_2: X\to X^*$ be two operators on the reflexive Banach
space $X$. If $T_1$ is monotone and hemicontinuous and $T_2$ is
strongly continuous, then $T_1+T_2$ is pseudomonotone.
\end{proposition}

The main existence result for pseudomonotone operators, see
\cite{Zeidler}, reads as follows:

\begin{theorem}\label{t1}
Assume $T: X\to X^*$ is pseudomonotone, bounded and coercive on the
real, separable, reflexive Banach space $X$. Then, the operator
equation $Tu=f$ has at least one solution $u\in X$ for every $f\in
X^*$.
\end{theorem}

We will apply Theorem \ref{t1} in order to prove the existence of
weak solutions $u\in W^{1,p}_0(\Omega)$ to (\ref{eq3}). To this end,
let $T, T_1, T_2 : W^{1,p}_0(\Omega)\to W^{-1,q}(\Omega)$ be the
operators defined as
$$
\langle Tu, v\rangle=\langle T_1 u,v\rangle+\langle T_2 u,v\rangle=
\int_{\Omega}a|Du|^{p-2}Du\cdot Dv\,
dx+\int_{\Omega}V|u|^{p-2}uv\,dx,
$$
for all $u,v\in W^{1,p}_0(\Omega)$, where
$\frac{1}{p}+\frac{1}{q}=1$.

We recall that $u\in W^{1,p}_0(\Omega)$ is said to be a weak
solution of (\ref{eq3}) if
\begin{equation}\label{eq4}
\int_{\Omega}a|Du|^{p-2}Du\cdot Dv\,
dx+\int_{\Omega}V|u|^{p-2}uv\,dx=\int_{\Omega}fv\,dx,
\end{equation}
for all $v\in W^{1,p}_0(\Omega)$. Furthermore, (\ref{eq4}) is
equivalent to the operator equation
$$
Tu=T_1 u+T_2 u=f, \quad u\in W^{1,p}_0(\Omega).
$$

\begin{theorem}\label{t2}
The equation
\begin{equation*}
-\text{div}\left(a|Du|^{p-2}Du\right) +V|u|^{p-2}u=f\quad \text{in
}\quad \Omega
\end{equation*}
has a solution $u\in W^{1,p}_0(\Omega)$ for every $f\in
W^{-1,q}(\Omega)$.
\end{theorem}

\begin{proof}
Recall that, for $1<p<\infty$, $W^{1,p}_0(\Omega)$ is a reflexive
Banach space. In order to apply the Br\'{e}zis' theorem, we have to
show that the operator $T$ is pseudomonotone, bounded and coercive
on $W^{1,p}_0(\Omega)$. It is well known that $T_1$ is monotone:
$$
\langle T_1 u - T_1 v,
u-v\rangle=\int_{\Omega}(a|Du|^{p-2}Du-a|Dv|^{p-2}Dv)\cdot (Du-Dv)\,
dx\geq 0.
$$
We also have
$$
\langle T_1 u,u \rangle=\int_{\omega}a|Du|^p\,dx\geq
M\|u\|^p_{W^{1,p}_0(\Omega)}.
$$
Thus, $T_1$ is coercive since
$$
\frac{\langle T_1 u,u \rangle}{\|u\|_{W^{1,p}_0(\Omega)}}\geq
M\|u\|_{W^{1,p}_0(\Omega)}^{p-1}\to +\infty \quad\text{ as }\quad
\|u\|_{W^{1,p}_0(\Omega)}\to+\infty.
$$
We now prove that $T_1$ is continuous. Let $F(\xi)=|\xi|^{p-2}\xi$
for $\xi\in\mathbb{R}$ and let $u_n\to u$ in $W^{1,p}_0(\Omega)$. By
the continuity of $F: L^p(\Omega)\to L^q(\Omega)$, it holds for
$1\leq i\leq N$ that $F(D_i u_n)\to F(D_i u)$ in $L^q(\Omega)$. The
H\"{o}lder inequality yields
\begin{eqnarray*}
   |\langle T_1 u_n - T_1 u, v\rangle|&=&\left|\int_{\Omega}\sum_{i=1}^N(aF(D_i u_n)-aF(D_i u))D_i v\,dx\right|  \\
   &\leq &\|a\|_{L^\infty(\Omega)}\sum_{i=1}^N\|F(D_i u_n)-F(D_i
   u)\|_{L^q(\Omega)}\|Dv\|_{L^p(\Omega)}.
\end{eqnarray*}
Therefore
$$
\|T_1 u_n - T_1 u\|_*=\sup_{\|v\|\leq 1}|\langle T_1 u_n - T_1
u,v\rangle|\leq C\sum_{i=1}^N\|F(D_i u_n)-F(D_i
   u)\|_{L^q(\Omega)}\to 0.
$$

As regards the boundedness of $T_1$, we have
\begin{eqnarray*}
  |\langle T_1 u,v\rangle| &=&\left| \int_{\Omega}a|Du|^{p-2}Du\cdot Dv\,
dx \right|\leq C \int_{\Omega}||Du|^{p-2}Du\cdot Dv|\,
dx\\
   &=& C\int_{\Omega}|Du|^{p-1}|Dv|\,
dx\leq C
\left(\int_{\Omega}|Du|^{p}\,dx\right)^\frac{1}{q}\left(\int_{\Omega}|Dv|^{p}\,dx\right)^\frac{1}{p}.
\end{eqnarray*}
Thus,
$$
\|T_1 u\|_*=\sup_{\|v\|_{W^{1,p}_0(\Omega)}\leq 1}|\langle T_1
u,v\rangle|\leq C\|u\|^{p/q}_{W^{1,p}_0(\Omega)}.
$$

Next, we consider $T_2$ and prove that it is strongly continuous.
Let $u_n\to u$ weakly in $W^{1,p}_0(\Omega)$. Then the
Rellich-Kondrachov theorem yields $u_n\to u$ in $L^p(\Omega)$. We
estimate
$$
|\langle T_2 u_n - T_2 u,v\rangle|\leq
\|V\|_{L^\infty(\Omega)}\int_\Omega|F(u_n)-F(u)||v|\,dx.
$$
Since $u_n\to u$ in $L^p(\Omega)$, it follows up to a subsequence
that, $u_n(x)\to u(x)$ a.e. in $\Omega$, and by continuity
$F(u_n)(x)\to F(u)(x)$  and $|F(u_n)(x)-F(u)(x)||v(x)|\to 0$ a .e.
in $\Omega$. The Egoroff theorem applies and gives us for
arbitrarily small $\beta$ a Lebesgue measurable set
$\Omega_\beta\subset \Omega$ with $|\Omega\setminus\Omega_\beta|\leq
\beta$ such that $|F(u_n)(x)-F(u)(x)||v(x)|\to 0$ uniformly in
$\Omega_\beta$. Denoting the characteristic function of
$\Omega_\beta$ by $\chi_{\Omega_\beta}$, a limit passage
($n\to\infty$) yields
$$
\int_{\Omega}\chi_{\Omega_\beta}|F(u_n)(x)-F(u)(x)||v(x)|\, dx\to 0.
$$
As the integral is uniformly bounded with respect to $n$ one may
send $\beta\to 0$ and get
$$
\|T_2 u_n-T_2 u\|_* =\sup_{\|v\|_{W^{1,p}_0(\Omega)}\leq 1}|\langle
T_2 u_n - T_2 u,v\rangle|\to 0,
$$
so $T_2$ is strongly continuous.

We now conclude with Proposition \ref{p1}: The operator $T$ is
pseudomonotone  as a strongly continuous perturbation of the
monotone operator $T_1$. The operator $T_2$ is bounded, since
strongly continuous. Hence $T=T_1+T_2$ is also bounded. But, as
$$
\inf_{x\in\mathbb{R}}V|x|^{p-2}xx=\inf_{x\in\mathbb{R}}V|x|^{p}=0,
$$
 we have that
$$
\langle T_2 u, u
\rangle=\int_\Omega(V|u|^{p-2}uu)\,dx\geq-\|V\|_{L^\infty(\Omega)}\|u\|^p_{L^p(\Omega)},
$$
for all $u\in W^{1,p}_0(\Omega)$. Hence $T=T_1+T_2$ is coercive
since $T_1$ is coercive. Theorems \ref{t1} applies with
$X=W^{1,p}_0(\Omega)$ and $X^*=W^{-1,q}(\Omega)$ and completes the
proof.
\end{proof}

\section{Existence of solution}\label{s3}
We will apply Theorem \ref{t2} from the previous section to prove
the existence of solution to (\ref{eq1}) and (\ref{eq2}). We need
some hypotheses on the coefficients.

\begin{itemize}
    \item [(H1)]The matrix $a$ with entries in $L^\infty(\mathbb{R}^N)$ is strictly positive definite, i.e. there exists
a constant $M>0$ such that
$$ \sum_{i,j=1}^N a_{ij}(y)\xi_j\xi_i\geq M|\xi|^2,
$$
for all $\xi\in \mathbb{R}^N$ and a.e. $y\in \mathbb{R}^N$.
    \item [(H2)] The matrix $a$ is periodic, i.e.
    $a_{ij}(y+ke_n)=a_{ij}(y)$, $1\leq i,j\leq N$, for
    $k\in\mathbb{Z}$ where $\{e_n\}_1^N$ is the canonical basis in
    $\mathbb{R}^N$.
    \item[(H3)] The function $V\in\mathcal{C}^\infty(\mathbb{R}^N)$ is $Y$-periodic and has mean value zero over $Y$.
\end{itemize}
\begin{remark}\label{r1}
For the existence theory it is sufficient to assume that $V\in
L^\infty(\mathbb{R}^N)$, but we will need the smoothness for the
asymptotic analysis.
\end{remark}

Recalling Theorem \ref{t2} we have the following:
\begin{corollary}\label{c1}
Under the hypotheses (H1)-(H3), problem (\ref{eq1}) possesses at
least one solution $u_\varepsilon\in W^{1,p}_0(\Omega)$ for every
$f\in W^{-1,q}(\Omega)$ and every $\varepsilon>0$.
\end{corollary}

We will postpone the discussion on the existence of solution to
problem (\ref{eq2}) until the last section, but it holds true that
\begin{corollary}\label{c1}
Under the hypotheses (H1)-(H3), problem (\ref{eq2}) possesses at
least one solution $(u,u_1)\in W^{1,p}_0(\Omega)\times
L^p(\Omega;W^{1,p}_{per}(Y))$ for every $f\in W^{-1,q}(\Omega)$.
\end{corollary}

\section{An a priori estimate}
The letter $E$ denotes in the sequel an ordinary sequence of
strictly positive real numbers $(0<\varepsilon<1)$ having $0$ as
accumulation point. Since we are interested in the asymptotic
behavior of the solutions to (\ref{eq1}) we need a uniform
(independent of $\varepsilon$) estimate of $u_\varepsilon$.

\begin{lemma}\label{l1}
Let $2\leq p<N$ and assume $a=a(y)$ is strictly positive definite,
bounded and periodic with period $Y=(0,1)^N$ and that $V=V(y)$ is
smooth, periodic and has vanishing mean value on $Y$. Then
$$
\|u_\varepsilon\|_{W^{1,p}_0(\Omega)}\leq C.
$$
\end{lemma}
\begin{proof}
Since  $V$ is smooth and we have $\int_Y V(y)\,dy=0$, there exists a
function $\Phi$ which is smooth and Y-periodic such that
\begin{equation*}
\left\{\begin{aligned}&\Delta_y \Phi(y)=V(y)\\&\Phi\quad
Y-\text{periodic}.
\end{aligned}\right.
\end{equation*}
Let us define $G(y)=D_y\Phi(y)$ and thus $\text{div}_y G(y)=V(y)$.
We observe that $|G|\leq C$ on bounded sets. Recalling
$F(z)=|z|^{p-2}z$, we can now write (\ref{eq1}) as
\begin{equation} \label{eq5}
\left\{\begin{aligned}&
-\text{div}\left(a(\frac{x}{\varepsilon})|Du_\varepsilon|^{p-2}Du_\varepsilon\right)
+\text{div}\left[G(\frac{x}{\varepsilon})\right]F(u_\varepsilon)=f\quad
\text{in }\quad
\Omega\\&u_\varepsilon=0\qquad\qquad\qquad\qquad\qquad\qquad\qquad\qquad\quad\quad\text{
on }
\partial\Omega,
\end{aligned}\right.
\end{equation}
Multiply the first equation in (\ref{eq5}) by $u_\varepsilon$ and
integrate over $\Omega$
$$
\int_\Omega
a(\frac{x}{\varepsilon})|Du_\varepsilon|^{p}\,dx=\int_\Omega
G(\frac{x}{\varepsilon})F'(u_\varepsilon)\cdot u_\varepsilon
Du_\varepsilon\,dx+\int_\Omega
G(\frac{x}{\varepsilon})F(u_\varepsilon)\cdot
Du_\varepsilon\,dx+\int_\Omega fu_\varepsilon\,dx.
$$
The uniform ellipticity of $a$ and the boundedness of $G$ yield
$$
M\int_\Omega |Du_\varepsilon|^{p}\,dx\leq C_1 \int_\Omega
|F'(u_\varepsilon)||Du_\varepsilon||u_\varepsilon|\,dx+C_2\int_\Omega
|F(u_\varepsilon)||Du_\varepsilon|\,dx+\|f\|_{W^{-1,q}(\Omega)}\|u_\varepsilon\|_{W^{1,p}_0(\Omega)}.
$$
We have
$$
|F(u_\varepsilon)|=|u_\varepsilon|^{p-1}\qquad\text{and}\qquad
|F'(u_\varepsilon)|\leq C|u_\varepsilon|^{p-2}.
$$
Therefore the first and second terms on the right-hand side can be
estimated as

$$C_1 \int_\Omega
|F'(u_\varepsilon)||Du_\varepsilon||u_\varepsilon|\,dx+C_2\int_\Omega
|F(u_\varepsilon)||Du_\varepsilon|\,dx\leq
C_3\int_\Omega|Du_\varepsilon||u_\varepsilon|^{p-1}\,dx.
$$
Let us now estimate
$$\int_\Omega|Du_\varepsilon||u_\varepsilon|^{p-1}\,dx.$$
Let $\delta>0$. By the Young inequality we have
$$
\int_\Omega|Du_\varepsilon||u_\varepsilon|^{p-1}\,dx\leq\frac{(\delta
p)^{-q/p}}{q}\int_\Omega
|u_\varepsilon|^p\,dx+\frac{\delta}{p}\int_\Omega|Du_\varepsilon|^p\,dx.
$$
For any integer $N\geq 3$ and any real number $2\leq p<N$, there
exists a real number $1<\eta<p$ such that $p\in[1,\eta^*]$, with
$\eta^*=\frac{N\eta}{N-\eta}$. For our purposes it is enough to show
that there exists $1<\eta<p$ such that for any given $N\geq 3$ and
any real number $2\leq p<N$ we have
$$
p\leq\frac{N\eta}{N-\eta}.
$$
This is equivalent to
$$
N\leq\frac{p\eta}{p-\eta}
$$
By putting $\eta=p-\alpha,\ \alpha>0$ we get
$$
N\leq\frac{p^2-p\alpha}{\alpha}.
$$
Choosing $\alpha>0$ small enough this inequality is valid for
arbitrary large integer $N\geq 3$ and any real number $2\leq p<N$.
Now, for $u_\varepsilon\in W^{1,\eta}_0(\Omega)$, the Poincareґ
inequality yields
$$
\int_\Omega |u_\varepsilon|^p\,dx\leq C_4\int_\Omega
|Du_\varepsilon|^\eta\,dx
$$
for $p\in[1,\eta^*]$, with $\eta^*=\frac{N\eta}{N-\eta}$. If we
chose $\delta$ so small that $\frac{C_3\delta}{p}<M$ we get
$$
\left(M-\frac{C_3\delta}{p}\right)\int_\Omega
|Du_\varepsilon|^p\,dx\leq C_4\frac{(\delta p)^{-q/p}}{q}\int_\Omega
|Du_\varepsilon|^\eta\,dx+C_5\|u_\varepsilon\|_{W^{1,p}_0(\Omega)}.
$$
If $\|u_\varepsilon\|_{W^{1,p}_0(\Omega)}<1$ we are done. Otherwise
we divide by $\|Du_\varepsilon\|^\eta$ to obtain
$$
\|u_\varepsilon\|^\alpha_{W^{1,p}_0(\Omega)}\leq C\qquad\text{for
}\quad \alpha=p-\eta>0.
$$
\end{proof}

\section{Two-scale convergence}\label{s5}
For the sake of completeness we summarize here the most important
facts. As a survey on the topic we refer to \cite{LNW}. We let
$\mathbb{T}^N$ denote the $N$-dimensional unit torus in
$\mathbb{R}^N$ and identify in the usual way $Y$-periodic functions
by those that are defined on $\mathbb{T}^N$ and introduce functions
$\varphi=\varphi(x,y)$ of Caratheodory type and consider their
traces $\varphi(x,\frac{x}{\varepsilon}),\ x\in\Omega$. A crucial
step is to construct an admissible class of test functions $\varphi$
such that for any bounded sequence
$\{u_\varepsilon\}_{\varepsilon\in E}\in L^2(\Omega)$ we have the
weak two-scale convergence
$$
\int_\Omega
u_\varepsilon(x)\varphi(x,\frac{x}{\varepsilon})\,dx\to\iint_{\Omega\times
Y}u(x,y)\varphi(x,y)\,dydx
$$
as $E'\ni\varepsilon\to 0$, $E'$ being a subsequence of $E$. Here
$u\in L^p(\Omega; L^p(Y))$. It turns out that
$L^q(\Omega;\mathcal{C}_{per}(Y))$ is the appropriate class of test
functions for which this works. If $u_\varepsilon$ is in addition
bounded in $W^{1,p}(\Omega)$, then we have the following convergence
of the gradients as $E'\ni\varepsilon\to 0$:
\begin{equation}\label{eq6}
\int_\Omega D_x
u_\varepsilon(x)\psi(x,\frac{x}{\varepsilon})\,dx\to\iint_{\Omega\times
Y}[D_xu(x)+D_y u_1(x,y)]\cdot\psi(x,y)\,dydx,
\end{equation}
where $u\in W^{1,p}(\Omega)$ and $u_1\in
L^p(\Omega;W^{1,p}_{per}(Y))$ for every $\psi\in
L^q(\Omega;\mathcal{C}_{per}(Y)^N)$. If in addition, the test
functions $\varphi\in L^q(\Omega;\mathcal{C}_{per}(Y))$ are chosen
to satisfies the centring condition
$$
\int_Y \varphi (\cdot,y)\,dy=0
$$
in the $y$-variable, then
\begin{equation}\label{eq7}
\int_\Omega
\frac{u_\varepsilon(x)}{\varepsilon}\varphi(x,\frac{x}{\varepsilon})\,dx\to\iint_{\Omega\times
Y}u_1(x,y)\varphi(x,y)\,dydx
\end{equation}
as $E'\ni\varepsilon\to 0$, see eg \cite{Douanla1, GW} for the proof
and \cite{Douanla2} for the corresponding result on periodic
surfaces.

\section{A convergence result}
The aim of this section is to prove a convergence result for
$$
\int_\Omega
\frac{F(u_\varepsilon(x))}{\varepsilon}\varphi(x,\frac{x}{\varepsilon})\,dx.
$$
AS we will see, it is essential that the test function $\varphi$
satisfies a centring condition in the local variable $y$, i.e. that
$$
\int_Y \varphi (\cdot,y)\,dy=0.
$$
The result of Lemma \ref{l2} will be crucial in our proof of the
homogenization of (\ref{eq1}) for the case $p>2$.

\begin{lemma}\label{l2}
The function $F$ is defined as above and $2<p<N$. Suppose that
$(u_\varepsilon)_{\varepsilon\in E}$ is bounded in
$W^{1,p}(\Omega)$, then there exist a subsequence $E'$ of $E$ and a
couple $(u,u_1)\in W^{1,p}(\Omega)\times
L^p(\Omega;W^{1,p}_{per}(Y))$ such that
$$
\int_\Omega
\frac{F(u_\varepsilon(x))}{\varepsilon}\varphi(x,\frac{x}{\varepsilon})\,dx\to\iint_{\Omega\times
Y}F'(u)u_1(x,y)\varphi(x,y)\,dydx
$$
as $E'\ni\varepsilon\to 0$, for all $\varphi=\varphi_1\varphi_2\in
\mathcal{C}^\infty_0(\Omega)\otimes \mathcal{C}^\infty_{per}(Y)$
with $\int_Y\varphi_2(y)\,dy=0$.
\end{lemma}
\begin{proof}
It is well-known that there exists a smooth $Y$-periodic solution
$\Phi$ to
\begin{equation*}
\left\{
\begin{aligned}
&\Delta_y \Phi(y)=\varphi_2(y),\\
&\Phi\ \ \ Y\text{-periodic}.
\end{aligned}
\right.
\end{equation*}
Let us define $\psi(y)=D_y\Phi(y)$ and thus
div$_y\psi(y)=\varphi_2(y)$. Then we can write
\begin{eqnarray*}
&&\int_\Omega
\frac{F(u_\varepsilon(x))}{\varepsilon}\varphi_1(x)\varphi_2(\frac{x}{\varepsilon})\,dx\\&&\qquad\quad=
 \int_\Omega F(u_\varepsilon(x))\varphi_1(x)\text{div}\left[\psi(\frac{x}{\varepsilon})\right]dx=-
 \int_\Omega D[F(u_\varepsilon(x))\varphi_1(x)]\cdot\psi(\frac{x}{\varepsilon})\,dx
 \\
   &&\qquad\quad=-\int_\Omega F'(u_\varepsilon(x))Du_\varepsilon\varphi_1(x)\cdot\psi(\frac{x}{\varepsilon})\,dx-
   \int_\Omega F(u_\varepsilon(x))D\varphi_1(x)\cdot\psi(\frac{x}{\varepsilon})\,dx.  \\
  \end{eqnarray*}
In order to pass to the limit in the first term, we need to prove
continuity with respect to $u$ for $F'(u)$. To this end we recall
that for $1<r< 2$ and for any real number $a,b$ it holds true that
$$
||a|^{r-2}a-|b|^{r-2}b|\leq C|a-b|^{r-1}.
$$
On letting $r=p-1$ for $2<p< 3$, using the definition of $F$ and the
above inequality, we conclude that $F'$ is H\"{o}lder continuous:
$$
|F'(u_1)-F'(u_2)|\leq C |u_1-u_2|^{p-2}.
$$
Integrating over $\Omega$ and using the H\"{o}lder inequality we
obtain for any $\varepsilon>0$,
\begin{equation}\label{eq8}
\|F'(u)-F'(u_\varepsilon)\|_{L^q(\Omega)}\leq C
\|u-u_\varepsilon\|^{p-2}_{L^p(\Omega)}.
\end{equation}
Likewise, For $r\geq 2$ and for real numbers $a,b$ it holds true
that
$$
||a|^{r-2}a-|b|^{r-2}b|\leq C(|a|+|b|)^{r-2}|a-b|.
$$
For $r=p-1$ with $3\leq p<N$, using the definition of $F$ and the
above inequality, we conclude
$$
|F'(u_1)-F'(u_2)|\leq C(|u_1|+|u_2|)^{p-3}|u_1-u_2|,
$$
so that $F'$ is Lipschitz continuous. Integrating over $\Omega$ and
using the H\"{o}lder inequality, we get for any $\varepsilon>0$
\begin{equation}\label{eq9}
\|F'(u)-F'(u_\varepsilon)\|_{L^q(\Omega)}\leq
C\|u-u_\varepsilon\|_{L^p(\Omega)},
\end{equation}
where we used the boundedness of the sequence
$(u_\varepsilon)_{\varepsilon\in E}$ in $L^p(\Omega)$. By the
Rellich-Kondrachov theorem, $u_\varepsilon$ converges strongly in
$L^p(\Omega)$ as $E'\ni \varepsilon\to 0$. Using now
(\ref{eq8})(\ref{eq9}) we conclude that $F'(u_\varepsilon)$
converges strongly in $L^q(\Omega)$ as $E'\ni \varepsilon\to 0$.
Combining this with the continuity and monotonicity properties of F
and the two-scale convergence property (\ref{eq6}) for gradients and
pass to the two-scale limit in both terms in
$$
-\int_\Omega
F'(u_\varepsilon(x))Du_\varepsilon\varphi_1(x)\cdot\psi(\frac{x}{\varepsilon})\,dx-
   \int_\Omega
   F(u_\varepsilon(x))D\varphi_1(x)\cdot\psi(\frac{x}{\varepsilon})\,dx,
$$
and obtain
\begin{eqnarray*}
  &&\lim_{E'\ni\varepsilon\to 0}\int_\Omega
\frac{F(u_\varepsilon(x))}{\varepsilon}\varphi_1(x)\varphi_2(\frac{x}{\varepsilon})\,dx\\&&\qquad=
-\int_{\Omega}\int_{Y} F'(u)[D_x u+D_y
u_1]\varphi_1(x)\cdot\psi(y)\,dydx- \int_{\Omega}\int_{Y} F(u)D_x
\varphi_1(x)\cdot\psi(y)\,dydx\\
&&\qquad=-\int_{\Omega}\int_{Y}F'(u)D_y
u_1\varphi_1(x)\cdot\psi(y)\,dxdy=\int_{\Omega}\int_{Y}F'(u)(x)u_1(x,y)\varphi_1(x)\text{div}_y\psi(y)\,dydx\\
&&\qquad=\int_{\Omega}\int_{Y}F'(u)(x)u_1(x,y)\varphi_1(x)\varphi_2(y)\,dydx.
\end{eqnarray*}
\end{proof}
\begin{remark}
Lemma \ref{l2} remains valid for more general functions $F$. For
instance we can choose $F$ to be continuous with H\"{o}lder
continuous derivative $F'$. For example $F(u)=|u|^p$, $1<p<1$.
\end{remark}

\section{The main theorem}
Let us recall the space $W^{1,p}_\#(Y)=\{v\in W^{1,p}_{per}(y):
\int_Y v(y)\,dy=0\}$. We can now state and prove the main
homogenization theorem for (\ref{eq1}).
\begin{theorem}\label{t3}
Under the assumption that $2\leq p<N$, there exist a subsequence
$E'$ of $E$ and functions $u\in W^{1,p}(\Omega)$ and $u_1\in
W^{1,p}_\#(Y)$ such that as $E'\ni \varepsilon\to 0$ the solutions
$u_\varepsilon\in W^{1,p}_0(\Omega)$ to (\ref{eq1}) satisfy:
$$
u_\varepsilon\to u \quad\text{in }\quad L^p(\Omega),
$$
$$
D_x u_\varepsilon\to D_x u+D_y u_1\quad\text{in }\quad
L^p(\Omega)\quad \text{(two-scale weakly),}
$$
where $(u, u_1)$ satisfies the two-scale homogenized system
(\ref{eq2}).
\end{theorem}
\begin{proof}
Let $\varphi_1\in\mathcal{D}(\Omega)$ and
$\varphi_2\in\mathcal{C}_{per}^\infty(Y)$ with
$\int_Y\varphi_2(y)dy=0$ and put
$\varphi_\varepsilon(x)=\varphi_1(x)+\varepsilon\varphi_2(\frac{x}{\varepsilon})$.
We then multiply (\ref{eq1}) by $\varphi_\varepsilon$ and integrate
over $\Omega$:
\begin{equation}\label{eq10}
\int_\Omega
a(\frac{x}{\varepsilon})|Du_\varepsilon|^{p-2}Du_\varepsilon \cdot
D\varphi_\varepsilon(x)\,dx+\int_\Omega\frac{1}{\varepsilon}V(\frac{x}{\varepsilon})
|u_\varepsilon|^{p-2}u_\varepsilon\varphi_\varepsilon(x)\,dx=\int_\Omega
f\varphi_\varepsilon(x)\,dx.
\end{equation}
By letting $\varphi_2\equiv 0$, a limit passage as $E'\ni
\varepsilon\to 0$, in (\ref{eq10}) using (\ref{eq7}) for the case
$p=2$ yields the two-scale homogenized equation
\begin{equation}\label{eq11}
    \int_\Omega\int_Y a(y)(D_x u+D_y u_1)\cdot
    D_x\varphi_1(x)\,dydx+\int_\Omega\int_Y
    V(y)u_1\varphi_1(x)\,dydx=\int_\Omega f\varphi_1(x)\,dx.
\end{equation}
By switching and letting instead $\varphi_1\equiv 0$ a limit passage
as $E'\ni \varepsilon\to 0$, in (\ref{eq1}) yields the local
two-scale homogenized equation
\begin{equation}\label{eq12}
    \int_\Omega\int_Y a(y)(D_x u+D_y u_1)\cdot
    D_y\varphi_2(y)\,dydx+\int_\Omega\int_Y
    V(y)u(x)\varphi_2(y)\,dydx=0
\end{equation}
By letting $\varphi_2\equiv 0$ a limit passage, $E'\ni\varepsilon\to
0$, in (\ref{eq10}) using Lemma \ref{l2} for $ 2<p<N $ yields the
global two-scale homogenized equation

\begin{eqnarray}
   &&\int_\Omega\int_Y a(y)|D_x u+D_y u_1|^{p-2}(
    D_x u+D_y u_1)\cdot D_x \varphi_1(x)\,dydx\nonumber\\
    &&\qquad+\int_\Omega\int_Y
    V(y)F'(u) u_1\varphi_1(x)\,dydx=\int_\Omega
    f\varphi_1(x)\,dx.\label{eq13}
\end{eqnarray}
By switching and letting instead $\varphi_1\equiv 0$ a limit
passage, $E'\ni\varepsilon\to 0$, in (\ref{eq1}) yields the local
two-scale homogenized equation
\begin{equation}\label{eq14}
\int_\Omega\int_Y a(y)|D_x u+D_y u_1|^{p-2}(
    D_x u+D_y u_1)\cdot D_y \varphi_2(y)\,dydx+\int_\Omega\int_Y
    V(y)F(u)\varphi_2(y)\,dydx=0.
\end{equation}
The differential form of (\ref{eq13}) and (\ref{eq14}) is
(\ref{eq2}), i.e.
\begin{equation*}
\left\{\begin{aligned}& -\text{div}_x\left(a(y)|D_x u+D_y
u_1|^{p-2}(D_x u+D_y u_1)\right)+V(y)F'(u)u_1=f,\\&
-\text{div}_y\left(a(y)|D_x u+D_y u_1|^{p-2}(D_x u+D_y
u_1)\right)+V(y)F(u)=0,\\&u=0\quad\text{ on }
\partial\Omega,
\end{aligned}\right.
\end{equation*}
where a straightforward differentiation for $p>2$ yields
$F'(u)=(p-1)|u|^{p-3}u$ for $u\geq 0$ and $F'(u)=(1-p)|u|^{p-3}u$
for $u< 0$.

\section{Homogenized equation}\label{s8}
For the linear case $p=2$ the local equation (\ref{eq12}) can be
written
\begin{equation}\label{eq15}
    -\text{div}_y (a(y)(D_x u(x)+D_y u_1(x,y)))+V(y)u(x)=0.
\end{equation}
By linearity we can make the usual ansatz
\begin{equation}\label{eq16}
    u_1(x,y)=D_x u(x)\cdot \chi(y).
\end{equation}
We can then write (\ref{eq15}) as
\begin{equation}\label{eq17}
    -\text{div}_y(a(y)D_y\chi(y))=\text{div}_y a(y)-V(y)u(x)[D_x
    u(x)]^{-1}\quad\text{in }\ \ \mathbb{T}^N.
\end{equation}
The right-hand side of (\ref{eq17}) has mean value zero over
$\mathbb{T}^N$ or equivalently over $Y$. Therefore, it is standard
that there exists a unique solution $\chi\in(H^1_\#(Y))^N$ to
(\ref{eq17}). Using (\ref{eq16}) the global homogenized equation can
be written as the convection-diffusion equation
\begin{equation}\label{eq18}
\left\{\begin{aligned}&
\text{div}(\overline{a}Du(x))+\overline{b}\cdot D u(x)=f(x)\  \
\text{in }\ \ \Omega\\&u(x)=0\ \qquad\quad\qquad\qquad\qquad\text{
on }
\partial\Omega,
\end{aligned}\right.
\end{equation}
with effective diffusion
$$
\overline{a}=\int_Y a(y)(I+D_y\chi(y))dy
$$
and effective convection
$$
\overline{b}=\int_Y V(y)\chi(y)dy.
$$
It is clear that (\ref{eq18}) has a unique solution $u\in
H^1_0(\Omega)$. However, if the effective convection field
$\overline{b}$ is large compared to the effective diffusion
$\overline{a}$ the homogenized equation (\ref{eq18}) is numerically
unstable. Let us now look at the nonlinear problem $2<p<N$. By the
nonlinearity one cannot separate variables like in the ansatz
(\ref{eq16}). The global and local two-scale homogenized systems are
coupled in this case. For the case when the principal term is linear
and the lower order term is Lipschitz continuous this is discussed
in \cite{AP2}. Let $x\in\Omega$ and let $\theta\in\mathbb{R}$ and
$\xi\in\mathbb{R}^N$ be fixed. We further let
$\mathcal{A}(\cdot,\xi)=a(\cdot)|\xi|^{p-2}\xi$ and introduce the
cell problem for the local parameter-dependent solution
$\chi(y)=\chi(x,\theta,\xi)(y)$:
\begin{equation}\label{eq19}
\left\{\begin{aligned}&
-\text{div}_y(\mathcal{A}(y,\xi+D_y\chi(y))=-V(y)F(\theta)\ \ \
\text{in} \ Y\\&\chi\in W^{1,p}_\#(Y).\end{aligned}\right.
\end{equation}
Since the right-hand side of (\ref{eq19}) has mean value zero over
$Y$, it is classical by using the theory of monotone elliptic
operators or direct methods in the calculus of variations that there
exists at least one solution $\chi\in W^{1,p}_\#(Y)$ to
(\ref{eq19}). Suppose now that $\chi_1$ and $\chi_2$ are two
solutions to (\ref{eq19}). We get
\begin{equation}\label{eq20}
    \int_Y
    (\mathcal{A}(y,\xi+D_y\chi_1(y))-\mathcal{A}(y,\xi+D_y\chi_2(y)))\cdot(D_y\chi_1(y)-D_y\chi_2(y))dy=0
\end{equation}
By the monotonicity of $\mathcal{A}$ with respect to the second
argument we conclude from (\ref{eq20}) that $D_y\chi_1=D_y\chi_2$ so
that $\chi_1$ and $\chi_2$ only differ by a constant as a function
of $y$. Using now the fact that $\chi_1$ and $\chi_2$ belong to
$W^{1,p}_\#(Y)$ we conclude that $\chi_1=\chi_2$, so that the
solution to (\ref{eq19}) is unique. If we now in particular choose
$\theta=u$ and $\xi=D_x u$ and let $x\in\Omega$ we have proved that
the solution $u_1=\chi(\cdot,u,D_x u)$ to (\ref{eq14}) is unique.

By the properties of $a$, $V$ and $F$ we can now repeat the
arguments for pseudomonotone operators and conclude that the
function $u\in W^{1,p}_0(\Omega)$ in the global two-scale
homogenized system (\ref{eq13})-(\ref{eq14}) is a weak solution to
the macroscopic homogenized equation
\begin{equation}\label{eq21}
\left\{\begin{aligned}& -\text{div}\left(\int_Y\mathcal{A}(y,D_x
u+D_y u_1)dy\right)+\left(\int_Y V(y)u_1dy\right)F'(u)=f\ \ \
\text{in} \ \Omega\\&u=0\qquad\qquad\qquad \qquad\qquad\qquad
\qquad\qquad\qquad\qquad\qquad\quad\ \text{on }
\partial\Omega .\end{aligned}\right.
\end{equation}
\begin{remark}Since we do not know the sign of $\int_Y V(y)u_1dy$ we can only
guarantee existence of solution to (\ref{eq21}). But if the solution
to the homogenized equation (\ref{eq21}) is unique, then the whole
sequence of solutions $u_\varepsilon$  to (\ref{eq1}) in Theorem
\ref{t3} converges to the solution u to (\ref{eq21}) in
$L^p(\Omega)$.
\end{remark}
\end{proof}

\end{document}